\documentclass[10pt]{article}

\usepackage[english]{babel}
\usepackage{amssymb}
\usepackage[all]{xy}
\usepackage{amsmath}
\usepackage{amsthm}
\usepackage{verbatim}

\newtheorem{theorem}{Theorem}[section]
\newtheorem{proposition}[theorem]{Proposition}
\newtheorem{lemma}[theorem]{Lemma}

\newtheorem{definition}[theorem]{Definition}
\newtheorem{remark}[theorem]{Remark}

\newcommand{\IF}{\mathbb{F}}

\newcommand{\ZZ}{\mathbb{Z}}

\begin{document}

\title{The maximum number of rational points for a genus $4$ curve over $\IF_{7}$ is $24$}

\author{Alessandra Rigato}

\date{}
\maketitle

\begin{abstract}
In this paper we show that the maximum number of rational points possible for a smooth, projective, absolutely irreducible genus $4$ curve over a finite field $\IF_{7}$ is $24$. It is known that a genus $4$ curve over $\IF_{7}$ can have at most $25$ points. In this paper we prove that such a curve can have at most $24$. On the other hand we provide an explicit example of a genus $4$ curve over $\IF_{7}$ having $24$ points.
\end{abstract}

\section{Introduction}

Given a prime power $q$ and a positive integer $g$ it has become an interesting challenge to determine the largest number $N_q(g)$ of rational points possible for a smooth, projective, absolutely irreducible genus $g$ curve over a finite field $\mathbb{F}_q$. Tables are constantly updated at \cite{manypoints}. The main result of this note is the following Theorem.

\begin{theorem}\label{teo: theorem}
The maximum number of points $N_{7}(4)$ for a genus $4$ curve defined over $\IF_{7}$ is $24$. Indeed:
\begin{enumerate}
\item Every genus $4$ curve defined over $\IF_{7}$ has at most $24$ rational points.
\item The projective, smooth, absolutely irreducible curve $C$ defined over $\IF_{7}$ by the set of affine equations
\begin{equation}\label{eq: C}
C: \,\left \{
\begin{array}{l}
y^{2} = x^{3}+3 \\
z^{2} = -x^3 + 3
\end{array}
\right.
\end{equation}
is a genus $4$ curve having $24$ rational points. Its Zeta function is
\[Z(t)=\frac{(7t^{2}+t+1)(7t^{2}+5t+1)^{3}}{(1-7t)(1-t)}.\]
\end{enumerate}
\end{theorem}

It is well known that an upper bound for $N_{7}(4)$ is $25$. This follows by Oesterl\'e's optimization of Serre's explicit formula bound \cite[Theorem 7.3]{schoof}. 
The bound of $25$ points was also obtained by Ihara \cite{ihara} and St\"ohr-Voloch \cite[Proposition 3.2]{stohrvoloch}. We show in Section \ref{sec: X25} that a genus $4$ curve over $\IF_{7}$ with $25$ rational points can not exist. This proves the first part of the Theorem. The second part is proved in Section \ref{sec: C24}. In Section \ref{sec: claim3} we present some properties and results on the Zeta function and the real Weil polynomial of a curve, while in Section \ref{sec: nonGalois} we introduce some notations and number theoretical results that will be useful for the study of non-Galois function fields extensions arising in Section \ref{sec: X25}.

\section{Zeta function and real Weil polynomial of a curve}\label{sec: claim3}

Many authors have recently focused on properties of the Zeta function and the real Weil polynomial of a curve in order to improve the bounds for the number of rational points of a curve over a finite field. The Zeta function of a curve $X$ defined over $\IF_{q}$ is given by
\[
Z(t)=\prod_{d\geq 1}\frac{1}{(1-t^d)^{a_d}},
\]
where $a_{d}$ denotes the number of places of degree $d$ of the function field of $X$. In particular $a_{1}=\#X(\IF_{q})$ is the number of rational places. If $X$ has genus $g$, its Zeta function $Z(t)$ is a rational function of the form
\[
Z(t) = \frac{L(t)}{(1-t)(1-qt)}, \quad\textrm{where}\quad L(t)  =  \prod_{i=1}^{g}(1-\alpha_{i} t)(1-\overline{\alpha_{i}} t)
\]
for certain $\alpha_{i} \in \mathbb{C}$ of absolute value $\sqrt{q}$. Therefore the Weil polynomial $L(t)=q^g t^{2g} + b_{2g-1}t^{2g-1}+\ldots + b_{1} t+1 \in \mathbb{Z}[t]$ is determined by the coefficients $b_{1}, \ldots, b_{g}$ which are in turn determined by the numbers $a_{1}, \ldots, a_{g}$ \cite[Section 5.1]{stichtenoth}. To a genus $g$ curve $X$ having $L(t)$ as numerator of its Zeta function, we associate the real Weil polynomial of $X$ defined by
\[
h(t)=\prod_{i=1}^{g}(t-\mu_i) \;\; \in\, \mathbb{Z}[t],
\]
where $\mu_{i}=\alpha_{i}+\overline{\alpha_{i}}$ is a real number in the interval $[-2\sqrt{q}, 2\sqrt{q}]$, for all $i=1,\ldots, g$. We have $L(t)=t^{g}h(qt+1/t)$. Moreover we denote by $a(X)=[a_{1},a_{2}, \ldots, a_{d}, \ldots]$ the vector whose $d$-th entry displays the number $a_{d}=a_{d}(X)$ of places of degree $d$ of the function field of $X$. Not every polynomial $h(t)$ with all zeros in the interval $[-2\sqrt{q}, 2\sqrt{q}]$ and with the property that
\[
\frac{L(t)}{(1-t)(1-qt)}=\prod_{d\ge 1}\frac{1}{(1-t^{d})^{a_{d}}}
\]
for certain integers $a_{d}\ge 0$ is necessarily the real Weil polynomial of a curve. The following result is due to Serre \cite[page Se 11]{serre}, \cite[Lemma 1]{lauter}.

\begin{proposition}  \label{prop: Res1}
Let $h(t)$ be the real Weil polynomial of a curve $X$ over $\IF_{q}$. Then $h(t)$ cannot be factored as $h(t) = h_{1}(t)h_{2}(t)$, with $h_{1}(t)$ and $h_{2}(t)$ non-constant polynomials in $\mathbb{Z}[t]$ such that the resultant of $h_{1}(t)$ and $h_{2}(t)$ is $\pm 1$.
\end{proposition}

Generalizations of this result are due to E.~Howe and K.~Lauter, for example \cite[Theorem 1, Proposition 13]{howelauter}:

\begin{proposition}\label{prop: HL}
Let $h(t) = (t - \mu)h_2(t)$ be the real Weil polynomial of a curve $X$ over $\IF_{q}$, where $t - \mu$ is the real Weil polynomial of an elliptic curve $E$ and $h_2(t)$ a non-constant polynomial in $\ZZ[t]$ coprime with $t-\mu$. If $r \neq \pm1$ is the resultant of $t-\mu$ and the radical of $h_2(t)$, then there is a map from $X$ to an elliptic curve isogenous to $E$, of degree dividing $r$.
\end{proposition}

\section{An explicit example of a genus $4$ curve having $24$ points over $\IF_{7}$}\label{sec: C24}

In this section we prove that the curve $C$ defined by the set of affine equations \eqref{eq: C} is a genus $4$ curve having $24$ rational points over $\IF_{7}$. 

\begin{proof}[Proof of Theorem \ref{teo: theorem} part 2.]
The function field of the smooth projective curve $C$ is $\IF_{7}(x,y,z)$, where $x$, $y$ and $z$ are defined by the set of equations \eqref{eq: C}. This is a Galois extension of the rational function field $\IF_{7}(x)$ of Galois group $G$ isomorphic to $\ZZ_{2}\times \ZZ_{2}$. The three quadratic subfields are: 
\begin{enumerate}
\item $\IF_{7}(x,y)$ the function field of the curve of affine equation
\begin{equation}\label{eq: Etilde}
\tilde{E}:\,y^{2}=x^{3}+3.
\end{equation}
This is a genus $1$ curve with $13$ rational points over $\IF_{7}$. The Weil polynomial is $7t^{2}+5t+1$.
\item $\IF_{7}(x,z)$ the function field of the curve of affine equation $z^{2}=-x^{3}+3$, which is isomorphic to $\tilde{E}$ and has hence the same Weil polynomial. 
\item $\IF_{7}(x,w)$, where $w=xy$ and hence $w^{2}=(x^{3}+3)(-x^{3}+3)=-x^{6}+2$. This is an affine equation of a smooth projective curve of genus $2$. One checks that for each $x \in \IF_{7}$ there are two rational points. Thus in total $14$ rational points, since the place at infinity has degree $2$. A small computation gives that there are $19$ places of degree $2$, thus the Weil polynomial is $(7t^{2}+t+1)(7t^{2}+5t+1)$. 
\end{enumerate}
Since $G$ is abelian we can compute the Zeta function $Z_{C}(t)$ of $C$ by means of the following \cite[Proposition 14.9]{rosen}
\[Z_{C}(t)=\prod_{\chi \in \widehat{G}}L(t,\chi),\]
where $\widehat{G}$ denotes the group of characters $\chi: G \to \{\pm 1\}$. For the non-trivial characters $\chi$, the L-function $L(t,\chi)=\prod_{P} (1-\chi(P)t^{\textrm{deg}\,P})^{-1}$ is precisely the Weil polinomial $L(t)$ of the curve corresponding to the quadratic function field fixed by $\textrm{ker}\,\chi$. The L-function of the trivial character is the Zeta function of $\mathbb{P}^{1}$. Thus the Weil polynomial of $C$ equals the product
\[(7t^{2}+t+1)(7t^{2}+5t+1)^{3}=7^{4}t^8 + \ldots + 118t^{2} + 16t + 1,\quad a(C)=[24, 3, 120, 558, \ldots] \] of the Weil polynomials of the quadratic function fields described above. From this it follows that the genus of $C$ is $4$ and that the number of rational points equals $24$.
\end{proof}

\begin{remark}
The elliptic curve $\tilde{E}$ defined by \eqref{eq: Etilde} is the unique genus $1$ optimal curve over $\IF_{7}$. Indeed, an optimal elliptic curve over $\IF_{7}$ has Frobenius polynomial $t^{2}+5t+7$. Hence, the discriminant is $-3$ and the curve admits an automorphism of order $3$. Therefore its Weierstrass equation is $y^{2}=x^{3}+b$ for some $b \in \IF_{7}^{*}$ \cite[Theorem 10.1]{silverman}. Only for $b=3$ one has a projective curve attaining the Hasse-Weil bound $q+1+g\sqrt{2q}$ of $13$ rational points.
\end{remark}

\section{Non-Galois extensions of degree $3$}\label{sec: nonGalois}

In this section we introduce some notation and adapt some results of \cite{rigato} to the examples in this paper. Let $E$ be an elliptic curve defined over $\IF_{q}$ and let $X$ be a genus $g$ curve over $\IF_{q}$. Let $X \to E$ be a morphism of degree $3$ such that the induced function field extension $\IF_{q}(X)/\IF_{q}(E)$ is non-Galois. 
 
\begin{definition}\label{def: defioverlineprime}
We denote by $\overline{X}$ the curve whose function field is the normal closure of $\IF_{q}(X)$ with respect to $\IF_{q}(E)$: it is a Galois extension of $\IF_{q}(E)$ having Galois group isomorphic to the symmetric group $S_3$. We denote by $X'$ the curve having as function field the quadratic extension of $\IF_{q}(E)$ corresponding to the group $A_3\simeq\mathbb{Z}_3$, the unique (normal) subgroup of $S_3$ of index $2$. The situation is described in the following picture:
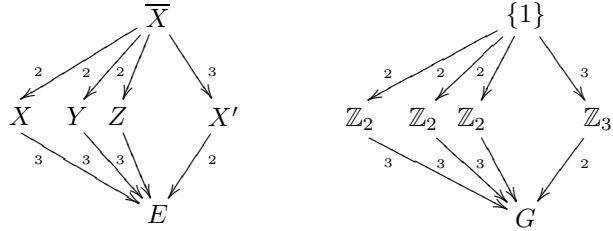
\begin{center}
$
\xymatrix@C=0.1pc { 
&&& \overline{X} \ar @{->}[dlll]_(.68){{\scriptscriptstyle 2}} \ar @{->}[dll]_(.67){{\scriptscriptstyle 2}} \ar @{->}[dl]_(.61){{\scriptscriptstyle 2}}  \ar@{->}[drr]^(.64){{\scriptscriptstyle 3}} && \\
{\qquad X} \ar @{->}[drrr]_(.32){{\scriptscriptstyle 3}} & {\;\;Y} \ar @{->}[drr]_(.33){{\scriptscriptstyle 3}} & **[r]Z \ar @{->}[dr]_(.39){{\scriptscriptstyle 3}} & && X' \ar @{->}[dll]^(.36){{\scriptscriptstyle 2}} \\
&&& E & &} 
\quad$
$
\xymatrix@C=0.1pc { 
&&& \{1\} \ar @{->}[dlll]_(.68){{\scriptscriptstyle 2}} \ar @{->}[dll]_(.67){{\scriptscriptstyle 2}} \ar @{->}[dl]_(.61){{\scriptscriptstyle 2}}  \ar@{->}[drr]^(.64){{\scriptscriptstyle 3}} && \\
{\qquad \mathbb{Z}_2} \ar @{->}[drrr]_(.32){{\scriptscriptstyle 3}} & {\;\;\mathbb{Z}_2} \ar @{->}[drr]_(.33){{\scriptscriptstyle 3}} & \mathbb{Z}_2 \ar @{->}[dr]_(.39){{\scriptscriptstyle 3}} & && \mathbb{Z}_3 \ar @{->}[dll]^(.36){{\scriptscriptstyle 2}} \\
&&& G & &} 
$
\end{center}
\end{definition}

In the rest of the note we will make often use of the following notation:

\begin{definition}\label{def: abcpointsE}
Consider a rational place $P$ of $\IF_{q}(E)$. We say that $P$ is
\begin{itemize}
\item[$a)$] an $A$-point of $E$, if $P$ splits completely in $\IF_{q}(X)$;
\item[$b)$] a $B$-point of $E$, if $P$ splits into two rational places of $\IF_{q}(X)$, one un\-ra\-mi\-fied and the other one with ramification index $2$;
\item[$c)$] a $B'$-point of $E$, if $P$ splits into two places of $\IF_{q}(X)$, one rational and the other one of degree $2$;
\item[$d)$] a $C$-point of $E$, if $P$\! is totally ramified in $\IF_{q}(X)$\! with ramification index $3$;
\item[$e)$] a $C'$-point of $E$, if $P$ is inert in $\IF_{q}(X)$ of degree $3$;
\end{itemize}
Moreover we denote by $a$, $b$, $b'$, $c$, $c'$ the number of $A$-points, $B$-points, $B'$-points, $C$-points and $C'$-points of $E$ respectively. 
\end{definition}

Here's an auxiliary lemma.
\begin{lemma}\label{lem: Cpoints}
Let $q \equiv 1 \;\textrm{mod}\,\; 3$ be a power of a prime $p\neq 2$. Then the local field $\IF_{q}((t))$ does not admit an extension $K$ of Galois group $G=Gal(K/\IF_{q}((t)))$ isomorphic to the symmetric group $S_{3}$.
\end{lemma}

\begin{proof}
We refer the reader to \cite[Chapter 4 \S 1]{serre1} for notations and results on local fields and their ramification groups. Suppose such an extension $K$ exists. Let $G_{0}$ be the inertia subgroup of $G$. Since $G_{0}$ is normal in $G$, one identifies the quotient $G/G_{0}$ with $Gal(K'/\IF_{q}((t)))$, where $K'$ is the largest unramified subextension of $K$ over $\IF_{q}((t))$. We have thus the following exact sequence
 \[ 1 \to Gal(K/K') \to Gal(K/\IF_{q}((t))) \to Gal(K'/\IF_{q}((t))) \to 1.\]
The quotient group $G/G_{0}$ is isomorphic to the Galois group of the residue field extensions. So it is cyclic. Therefore the field $K$ is ramified over $\IF_{q}((t))$. Moreover $K$ cannot be totally ramified, because
being tamely ramified, such an extension would be cyclic. Hence the inertia group is isomorphic to $\ZZ_{3}$ the unique normal subgroup of $S_{3}$. Thus $G$ is a semidirect product of $\ZZ_{3}$ and $\ZZ_{2}$:
\[ G=\langle \sigma, \tau \!: \sigma^{2}=1, \,\tau^{3}=1, \,\sigma \tau\sigma^{-1}=\tau^{q}\rangle,\]
where $\tau$ is a generator of $Gal(K/K')$ and $\sigma$ a lift of Frobenius of order $2$ generating $Gal(K'/\IF_{q}((t)))$. Now the order of $\tau$ is $3$ and $q \equiv 1 \;\textrm{mod}\,\; 3$, hence $G$ is abelian. But this is a contradiction.
\end{proof}

Assume that the field of definition of the curve $\overline{X}$ is $\IF_{q}$, then the following holds:
\begin{lemma}\label{lemma: abcpointsE}
\mbox{}
\begin{itemize}
\item[$a)$] The $A$-points of $E$ split completely in $\IF_{q}(\overline{X})$ and in $\IF_{q}(X')$.
\item[$b)$] The $B$-points of $E$ split in $\IF_{q}(\overline{X})$ into three rational places having ramification index $2$, while they are ramified in $\IF_{q}(X')$.
\item[$c)$] The $B'$-points of $E$ split in $\IF_{q}(\overline{X})$ into three rational places having relative degree $2$, while they are inert of degree $2$ in $\IF_{q}(X')$.
\item[$d)$] The $C$-points of $E$ are not totally ramified in $\IF_{q}(\overline{X})$. Moreover, if $q \equiv 1 \;\textrm{mod}\,\; 3$, they split in $\IF_{q}(\overline{X})$ into two rational places having ramification index $3$, while they split in $\IF_{q}(X')$ into two unramified rational places.
\item[$e)$] The $C'$-points of $E$ split in $\IF_{q}(\overline{X})$ into two places having relative degree $3$, while they split into two rational places in $\IF_{q}(X')$.
\end{itemize}
\end{lemma}

\begin{proof}
\mbox{}
\begin{itemize}
\item[$a)$] The rational places of $\IF_{q}(E)$ splitting completely in $\IF_{q}(X)$ split completely in the isomorphic function field $\IF_{q}(Y)$ as well. Hence they split completely in the compositum $\IF_{q}(\overline{X})$ and thus in the function field of $X'$ too.
\item[$b)$] Let $P$ be a $B$-point of $E$. Then the number of places of $\IF_{q}(\overline{X})$ lying over the place $P$ of $\IF_{q}(E)$ must be greater than or equal to $2$ and divide $6$. Moreover each of them must have (the same) ramification index $e\geq 2$ and dividing $6$ since the extension is Galois. 
\item[$c)$] The reasoning is analogous to the one above for the $B$-points.
\item[$d)$] The rational places of $\IF_{q}(E)$ that are totally ramified in $\IF_{q}(X)$ are also totally ramified in the isomorphic function field $\IF_{q}(Y)$. Since $\textrm{char}\, \IF_{q}\neq 3$ the ramification is tame, thus they are ramified in the compositum $\IF_{q}(\overline{X})$ with the same ramification index. Hence a place $Q$ of $\IF_{q}(\overline{X})$, lying over $P$ has ramification index $3$. Moreover, suppose $q \equiv 1 \;\textrm{mod}\, \;3$. The degree of $Q$ can be either $2$ or $1$. Lemma \ref{lem: Cpoints} shows that the first case is impossible. In the latter case, there exists another rational place $Q'$ lying over $P$ and having ramification index $3$. This concludes the proof.
\item[$e)$] Let $P$ be a rational place of $\IF_{q}(E)$ inert of relative degree $3$ in $\IF_{q}(X)$. Since $P$ is not ramified neither in $\IF_{q}(X)$ nor in the isomorphic function field $\IF_{q}(Y)$, the place $P$ is not ramified in the compositum $\IF_{q}(\overline{X})$. Since $P$ unramified, its decomposition group must be cyclic. Hence it must have order $3$. So there are two places of relative degree $3$ of $\IF_{q}(\overline{X})$ over $P$ and hence two rational places of $\IF_{q}(X')$ over $P$.
\end{itemize}\vspace{-0.7cm}
\end{proof}

\begin{remark}\label{rem: B'C'type}
In the case of the $B'$-points and the $C$-points the arguments used do not depend on the fact that these ramifying places are rational. Hence the same results hold for higher degree places.
\end{remark}

\section{Non-existence of a genus $4$ curve with $25$ points over $\IF_{7}$}\label{sec: X25}

Let $X$ be a genus $4$ curve having $25$ rational points over $\IF_{7}$. In this section we prove that such a curve $X$ can not exist.

\begin{proposition}\label{prop: deg3coveringE10}
The function field $\IF_{7}(X)$ of $X$ is a degree $3$ extension of the function field $\IF_{7}(E)$ of an elliptic curve $E$ having $10$ rational points over $\IF_{7}$.
\end{proposition}

\begin{proof}
We can compute a list of monic degree $4$ polynomials $h(t)\in \mathbb{Z}[t]$ for which $a_{1}=25$, $a_{d}\ge 0$ for $d\ge 2$ in the relation described in Section \ref{sec: claim3}. Moreover we require that $h(t)$ has all zeros in the interval $[-2\sqrt{7},2\sqrt{7}]$ and that the conditions of Proposition \ref{prop: Res1} are satisfied. A short computer calculation gives that if a genus $4$ curve $X$ having $25$ rational points over $\IF_{7}$ exists there is a unique possibility for its real Weil polynomial, namely
\begin{equation}\label{eq: realWeil}
h(t)=(t + 2)(t + 5)^{3}\quad \textrm{with} \quad a(X)=[25, 1, 115, 576,\ldots].
\end{equation}
Since the resultant of $t+2$ and $t+5$ is $3$, Proposition \ref{prop: HL} applies. Notice that $t+2$ is the real Weil polynomial of a genus $1$ curve with $10$ rational points over $\IF_{7}$.
\end{proof}
 
Let $\IF_7(X)/\IF_7(E)$ be as in Proposition \ref{prop: deg3coveringE10}.

\begin{lemma}\label{lemma: Galois}
The degree $3$ function field extension $\IF_{7}(X)/\IF_{7}(E)$ is not Galois.
\end{lemma}
  
\begin{proof}
Suppose that $\IF_{7}(X)/\IF_{7}(E)$ is Galois. Since $E$ has $10$ rational points and $X$ has $25$, the only possibility for the splitting behavior of the rational places of $\IF_{7}(E)$ is that eight places split completely in $\IF_{7}(X)$, one place $P$ is (tamely) totally ramified and one place $T$ is inert, i.e. it gives rise to a place of degree $3$ in $\IF_{7}(X)$. Moreover, by the Hurwitz genus formula the degree of the different $D$ of $\IF_{7}(X)/\IF_{7}(E)$ is $2\cdot4-2=6$. Hence we can only have that $D=2P+2Q$, where $Q$ is a place of $\IF_{7}(E)$ of degree $2$. The function field of $X$ is hence a subfield of the ray class field of $\IF_{7}(E)$ of conductor $P+Q$, where all rational places in $E(\IF_{7})\backslash \{P,T\}$ split completely. By translation we can always assume that $P$ equals the point at infinity of $E$ and using the elliptic involution we can let $T$ vary among only half of $E(\IF_{7})\backslash \{P\}$. Up to isomorphism there are two elliptic curves over $\IF_{7}$ with $10$ rational points. They have affine equations $y^{2}=x^{3}+x+4$ and $y^{2}=x^{3}+3x+4$. A short MAGMA computation allows to conclude that this ray class field is trivial for both curves $E$. See the appendix for the MAGMA code.
\end{proof}

Since the extension $\IF_{7}(X)/\IF_{7}(E)$ is non-Galois, we may apply Lemma \ref{lemma: abcpointsE}. In the following Lemma we present the possibilities for the numbers $a$, $b$, $\ldots$ of $A$-points, $B$-points, $\ldots$ of $E$. Moreover we consider the curve $\overline{X}$ whose function field is the Galois closure of $\IF_{7}(X)$ with respect to $\IF_{7}(E)$. We compute its number $\overline{N}$ of rational points and its genus $\overline{g}$ in each case. 

\begin{lemma}\label{lemma: nonGalois}
There are five possibilities for the splitting behavior of the rational places of $\IF_{7}(E)$ in $\IF_{7}(X)$. The curve $\overline{X}$ is defined over $\IF_{7}$ in any of these cases. Its genus $\overline{g}$ and its number $\overline{N}$ of rational points are displayed in Table \ref{tab: tab1}.
\begin{table}[h]
\begin{center}
\begin{tabular}{| c | c | c | c | c | c | c | c |}
\hline
$\textrm{case}$ & $a$ & $b$ & $b'$ & $c$ & $c'$ & $\stackrel{\phantom{\_}}{\overline{N}}$ & $\overline{g}$\\
\hline 
$\textrm{I}$ &  $8$ & $0$ & $1$ & $0$ & $1$ & $48$ & $10$\\ 
\hline
$\textrm{II}$ &  $8$ & $0$ & $0$ & $1$ & $1$ & $50$ & $7\;\, \textrm{or}\;\,9$\\
\hline 
$\textrm{III}$ &  $7$ & $2$ & $0$ & $0$ & $1$ & $48$ & $8\;\, \textrm{or}\;\,10$\\
\hline 
$\textrm{IV}$ &  $7$ & $1$ & $1$ & $1$ & $0$ & $47$ & $9$ \\
\hline
$\textrm{V}$ &  $6$ & $3$ & $1$ & $0$ & $0$ & $45$ & $10$\\
\hline 
\end{tabular}
\caption{Splitting behavior of the rational places of $\IF_{7}(E)$ in $\IF_{7}(X)$}\label{tab: tab1}
\end{center}
\end{table}
\end{lemma}

\begin{proof}
The curve $X$ has $25$ rational points over $\IF_{7}$, while the curve $E$ has $10$. Hence by Lemma \ref{lemma: abcpointsE} we have that the numbers $a$, $b$, $\ldots$ of $A$-points, $B$-points, $\ldots$ of $E$ must satisfy:
\begin{equation}\label{eq: set1}
\left \{
\begin{array}{l}
3a+2b+b'+c = 25\\
a+b+b'+c+c' =10
\end{array}
\right.
\end{equation}
The different $D$ of the function field extension $\IF_{7}(X)/\IF_{7}(E)$ has degree $6$ by the Hurwitz formula. Moreover the contribution of all ramifying rational points can not be exactly $5$. Hence we also have that 
\begin{equation}\label{eq: set2}
\left \{
\begin{array}{l}
b+2c \leq 6\\
b+2c\neq 5.
\end{array}
\right.
\end{equation}
The values of $a$, $b$, $\ldots$ that satisfy both \eqref{eq: set1} and \eqref{eq: set2} are those displayed in Table \ref{tab: tab1} plus the values $(a,b,b',c,c')=(7,1,2,0,0)$. But this case can never occur: $b'=2$ implies that $\IF_{7}(X)$ has at least two places of degree $2$. Which contradicts the fact that $a_{2}(X)=1$ as in \eqref{eq: realWeil}. We remark that $\overline{X}$ is indeed defined over $\IF_{7}$ since in any case the number $a$ of $A$-points of $E$ in non-zero (the full constant field of the function field of $\overline{X}$ is always contained in the residue fields of its places). The number $\overline{N}$ of $\IF_{7}$-rational points of $\overline{X}$ is $6a+3b+2c$ by Lemma \ref{lemma: abcpointsE}. The genus $\overline{g}$ of $\overline{X}$ is computed by means of the Hurwitz formula $2\overline{g}-2=\textrm{deg}\, \overline{D}$, where $\overline{D}$ is the different of the function field extension $\IF_{7}(\overline{X})/\IF_{7}(E)$. We determine the degree of the divisor $\overline{D}$ as follows.\\
The ramifying rational places of $\IF_{7}(E)$ give a contribution of $b+2c$ to the degree of the different $D$. Since $a_{2}(X)=1$, there can be at most one degree $2$ place of $\IF_{7}(E)$ that (totally) ramifies in $\IF_{7}(X)$. If there is such a ramified place, than $b'=0$: by Lemma \ref{lemma: abcpointsE} and Remark \ref{rem: B'C'type}, the degree of $\overline{D}$ is $12$ in case $II$ and $14$ in case $III$. Thus we have $\overline{g}=7$ and $\overline{g}=8$ respectively. In case there is no such a ramifying place, there always exists a unique place of $\IF_{7}(E)$ of degree $6-(b+2c)$ splitting in $\IF_{7}(X)$ into two places of the same degree, one having ramification index $2$ and one unramified. In case $I$ there is moreover the possibility that two places of $\IF_{7}(E)$ of degree $(6-(b+2c))/2=3$ appear in the support of $D$, each of them splitting in $\IF_{7}(X)$ into two places of the same degree, one having ramification index $2$ and one unramified. By Lemma \ref{lemma: abcpointsE} and Remark \ref{rem: B'C'type}, the degree of $\overline{D}$ is $18$ for both possibilities in case $I$ and thus $\overline{g}=10$. The degree of $\overline{D}$ is $16$ in cases $II$ and $IV$ and it is $18$ in cases $III$ and $V$, giving $\overline{g}=9$ and $\overline{g}=10$ respectively.
\end{proof}

\begin{lemma}\label{lem: noXbar}
The does not exist a curve $\overline{X}$ with genus $\overline{g}$ and number $\overline{N}$ of $\IF_{7}$-rational points as displayed in Table \ref{tab: tab1}.
\end{lemma}

\begin{proof}
By the Oesterl\'e bound \cite[Theorem 7.3]{schoof} a curve having $48$ (resp. $45$) rational points over $\IF_{7}$ must have genus at least $11$ (resp. $10$). Hence in the first four cases the curve $\overline{X}$ can not exist because it has too many points for its genus. In $\textrm{case V}$ the curve $\overline{X}$ has exactly $45$ rational points. Moreover, since $b'=1$, the curve $\overline{X}$ has at least three places of degree $2$ by Lemma \ref{lemma: abcpointsE}. We search for the real Weil polynomial $h(t)$ of such a curve. See Section \ref{sec: claim3} for the relation between $h(t)$ and the coefficients $a_{d}$ of the Zeta function of $\overline{X}$. We compute a list of monic degree $4$ polynomials $h(t)$ with integer coefficients for which $a_{1}=25$, $a_{2}\ge 3$ and $a_{d}\ge 0$ for $d\ge 3$. Moreover we require that $h(t)$ has all zeros in the interval $[-2\sqrt{7},2\sqrt{7}]$. A short computer calculation gives that there is only one such a polynomial, namely
\[ h(t)=(t+3)^{3}(t+4)^{7}  \quad \textrm{with} \quad a(\overline{X})=[45, 3, 17, 807, \ldots]. \]
But since the resultant of $t+3$ and $t+4$ is $1$, we have a contradiction by Proposition \ref{prop: Res1} and hence also in this case the curve $\overline{X}$ does not exist.
\end{proof}

Summing up these results we prove the first part of the Theorem.

\begin{proof}[Proof of Theorem \ref{teo: theorem} part 1.]
Let $X$ denote a genus $4$ curve over $\IF_{7}$. We pointed out in the Introduction that an upper bound for the number of rational points of $X$ is $25$. If a curve $X$ with $25$ rational points exists, then Proposition \ref{prop: deg3coveringE10} implies that it is a degree $3$ covering of an elliptic curve $E$ with $10$ rational points. On the other hand, Lemmas \ref{lemma: Galois}, \ref{lemma: nonGalois} and \ref{lem: noXbar} show that the curve $\overline{X}$ and hence the curve $X$ can not exist. This proves that every genus $4$ curve over $\IF_{7}$ has at most $24$ rational points.
\end{proof}

\section{Appendix}

We list here the MAGMA code used for the ray class field computation in Lemma \ref{lemma: Galois}.
\begin{verbatim}
kx<x> := FunctionField(GF(7));
kxy<y> := PolynomialRing(kx);
E:=FunctionField(y^2-x^3-x-4);  
     // alternatively E:=FunctionField(y^2-x^3-3*x-4);  
Genus(E);
P:=Places(E,1);
print  "Rational places of E:  ",P;
Q:=Places(E,2); #Q;
for i in {2, 4, 5, 8, 9} do             
     // this sets the rational place T to be (x,y+2), (x+5,y),
     // (x+1,y+3), (x+3,y+x), (x+2,y+1)
for j:=1 to #Q do
D:=1*P[1]+1*Q[j];
     // P[1] is the place at infinity
S:={2,3,4,5,6,7,8,9,10} diff {i};
      // set of splitting places
R, mR := RayClassGroup(D);
U := sub<R | [P[x]@@mR : x in S]>;
if not (#quo<R|U> eq 1) then
print "********************************************************";
quo<R|U>;
C := FunctionField(AbelianExtension(D, U)); C;
print "Genus", Genus(C);
print "Number of places a(C)=[",#Places(C, 1),",",#Places(C, 2),
                     ",",#Places(C, 3),",",#Places(C, 4),",...]";
print "Degree 2 place of E ramifying is Q=", Q[j]; 
print "Inert place of E is T=", P[i];
end if; end for; end for;
\end{verbatim}

\vfill

\noindent
\textsc{Alessandra Rigato}\\
K.U.~$\!$Leuven, Department of Mathematics,\\
Celestijnenlaan 200 B, B-3001 Leuven (Heverlee), Belgium\\
\verb+Alessandra.Rigato@wis.kuleuven.be+

\end{document}